\documentclass[12pt]{article}
\usepackage{amsmath, amsthm, amscd, amsfonts, amssymb, graphicx}
\usepackage{titles}
\usepackage{algorithm}
\usepackage{algpseudocode}
\usepackage{subfigure} 
\usepackage{caption}
\usepackage{graphics,graphicx}
\usepackage{amsmath, amssymb, amsthm}
\usepackage{mathrsfs} 
\usepackage{enumitem}

\usepackage{indentfirst}

\numberwithin{equation}{section}
\usepackage{epsfig}
\usepackage{color}
\usepackage{fancyhdr}
\usepackage{titlesec}
\usepackage{tikz}
\usepackage{afterpage}
\usetikzlibrary{arrows}
\usetikzlibrary{decorations.markings}
\tikzstyle{block}=[draw opacity=0.7,line width=1.4cm]
\tikzset{
	big black arrow/.style={
		decoration={markings,mark=at position 1 with {\arrow[scale=2.5,black]{>}}},
		postaction={decorate},
		shorten >=0.4pt},
	line/.style={draw, ->}}

\usepackage{eufrak}
\setbox0=\hbox{$+$}
\newdimen\plusheight
\plusheight=\ht0
\def\+{\;\lower\plusheight\hbox{$+$}\;}

\setbox0=\hbox{$-$}
\newdimen\minusheight
\minusheight=\ht0
\def\-{\;\lower\minusheight\hbox{$-$}\;}

\setbox0=\hbox{$\cdots$}
\newdimen\cdotsheight
\cdotsheight=\plusheight
\def\cds{\lower\cdotsheight\hbox{$\cdots$}}
\newtheorem{theorem}{Theorem}[section]
\newtheorem{lemma}[theorem]{Lemma}
\newtheorem{corollary}[theorem]{Corollary}

\theoremstyle{definition}

\theoremstyle{remark}
\newtheorem{remark}[theorem]{Remark}
\numberwithin{equation}{section}

\allowdisplaybreaks
\title{\textbf{Analytic proofs of  Andrews-Bachraoui identities related to two-color partitions with evens in one color}}
\date{}

\begin{document}
\maketitle

\vspace{-2cm}

\begin{center}
	{\bf Gaurab Bardhan$^1$ and Nipen Saikia$^{2, \ast}$}\\
	$^1$Department of Mathematics, Tyagbir Hem Baruah College,\\ Jamugurihat, Sonitpur, Assam, India.\\
	E. Mail: gaurabbardhan561@gmail.com
	\vskip2mm
	$^2$Department of Mathematics, Rajiv Gandhi
	University,\\ Rono Hills, Doimukh, Arunachal Pradesh, India.\\
	E. Mail(s): nipennak@yahoo.com\\
	$^\ast$\textit{Corresponding author}.\end{center}
\begin{abstract}Andrews and Bachraoui (\textit{Int. J. Number Theory} (2026)) studied the two-color partition function $F(n)$ of a non-negative integer $n$ wherein odd parts may appear in two colors (red and blue) and even parts appear in one color (blue). For any non-negative integer $n,$ they also considered some restricted versions of $F(n)$: $F_0(n)$: the number of partitions of $n$ counted by $F(n)$ such that the number of odd parts  in red color is even; $F_1(n)$: the number of partitions counted by $F(n)$ such that the number of odd parts in red color is odd; $H(n)$: the number of partitions of $n$  counted by $F(n)$ such that the parts of the same color do not repeat. The main purpose of this paper is to present the analytic proofs of the $q$-series identities connected with $F(n)$ and $H(n),$ which appeared as open problems in the original paper. We also prove some congruences of 
$F_0(n)$ and $F_1(n)$ modulo $2,$ $4,$  and $8$ by using $ q $-series.
\end{abstract}

 \vskip2mm


\section{Introduction}Throughout the paper, we will use the following standard $q$-series notations from \cite{andrews, gas}. For any complex number $a,$ $a_1,\;a_2,\;\ldots,\;a_k$ and $q$ with $|q|<1$, 
$$
(a;q)_0=1,\qquad
(a;q)_n=\prod_{j=0}^{n-1}(1-aq^j),\qquad
(a;q)_\infty=\prod_{j=0}^{\infty}(1-aq^j).\notag
$$
$$
(a_1,\ldots,a_k;q)_n=\prod_{j=1}^{k}(a_j;q)_n
\qquad\text{and}\qquad 
(a_1,\ldots,a_k;q)_\infty=\prod_{j=1}^{k}(a_j;q)_\infty.
$$
For brevity, we will use the notation $f_t:=(q^t;q^t)_\infty$
for any positive integer $t$.

A non-increasing finite sequence of positive integers
$
\lambda_1\geq \lambda_2\geq\cdots\geq \lambda_k>0
$
is said to be a partition of a positive integer $n$ if
$
n=\lambda_1+\lambda_2+\cdots+\lambda_k.
$ The positive integers $\lambda_i$ are called parts of the partition.
Euler \cite{euler1748introductio} gave the generating function of $p(n)$ as
$$
  \sum_{n=0}^{\infty} p(n) q^n = \dfrac{1}{(q; q)_\infty}; \qquad p(0)=1.\notag 
$$

Recently, Andrews and Bachraoui  \cite{2clr} (see also \cite{arx})  studied the two-color partition function \( F(n) \) of a non-negative integer \( n \) wherein odd parts may appear in two colors (red and blue) and even parts appear in one color (blue).  For more details on two-color partitions, see \cite{and1987,and2021,andbach,andkumar} and references therein. 
Depending on parity and colors of parts, Andrews and Bachraoui in \cite{2clr}  considered some restricted versions of \( F(n) \). In this paper, we are concerned with the following versions of $F(n)$:\\
$\bullet$ \( F_0(n) \): the number of partitions of \( n \) counted by \( F(n) \) such that the number of odd parts in red color is even; \\
$\bullet$ \( F_1(n)\): the number of partitions counted by \( F(n) \) such that the number of odd parts in red color is odd; \\
$\bullet$ \( H(n) \): the number of partitions of \( n \) counted by \( F(n) \) such that the parts of the same color do not repeat. 

The first objective of this paper is to present  analytic proofs of two \(q\)-series identities connected with the partition functions  \( F(n) \) and \( H(n) \), which appeared as open problems in  \cite[pp. 10--11, (5.3) and (5.6)] {2clr}, respectively:
\begin{align}
	&\hspace{-2cm}\sum_{n=0}^{\infty}
	\dfrac{q^{2n}}{(q^2;q^2)_n(q;q^2)_n^2}
	\left(
	1+\dfrac{q}{1-q^{2n+1}}
	+\dfrac{q}{(1-q^{2n+1})^2}
	\right)\notag\\
	&\hspace{-2cm}=1+\sum_{n=0}^{\infty}
	\dfrac{q^{2n+1}}{(q^{2n+2};q^2)_\infty(q^{2n+3};q^2)_\infty^2}
	\left(
	q+\dfrac{1}{1-q^{2n+1}}
	+\dfrac{1}{(1-q^{2n+1})^2}
	\right)\notag\\
	\label{OP1}&\hspace{-2cm}=
	\dfrac{1}{(q^2;q^2)_\infty(q;q^2)_\infty^2}
\end{align}    and 
\begin{align}
	&\sum_{n=0}^{\infty}
	q^{2n+1}(-q^2;q^2)_n(-q;q^2)_{n+1}^2
	\left(
	q+\dfrac{1}{1+q^{2n+1}}
	+\dfrac{1}{(1+q^{2n+1})^2}
	\right)\notag\\
	&=\sum_{n=0}^{\infty}
	q^{2n+1}(-q^{2n+1},-q^{2n+2},-q^{2n+3};q^2)_\infty
	\left(
	1+\dfrac{1}{1+q^{2n+1}}
	+\dfrac{q}{(1+q^{2n+1})(1+q^{2n+2})}
	\right)\notag\\
	\label{2op5.6}&=
	(-q^2;q^2)_\infty(-q;q^2)_\infty^2-1.
\end{align}

The second objective is to prove some congruences modulo \( 2\), \( 4\),  and \( 8 \)  for the partition functions  \( F_0(n) \) and \( F_1(n) \) by using some theta-functions and \( q \)-series. 

From \cite[Theorems 1 and 2]{arx}, the generating functions of \( F_0(n) \) and \( F_1(n) \) are given, respectively, as
\begin{equation}\label{I3}
	\sum_{n=0}^{\infty}F_0(n)q^n
	=
	\dfrac{(q^{16},-q^6,-q^{10};q^{16})_\infty}
	{(q;q^2)_\infty(q^2;q^2)^2_\infty}
\end{equation}
and
\begin{equation}\label{I4}
	\sum_{n=0}^{\infty}F_1(n)q^n
	=
	\dfrac{q(q^{16},-q^2,-q^{14};q^{16})_\infty}
	{(q;q^2)_\infty(q^2;q^2)^2_\infty}.
\end{equation}
\begin{remark}\it
$(i)$~It is observed that the generating functions of \( F_0(n) \) and \( F_1(n) \) given in Theorems 1.3 and 1.5  of \cite[p. 2]{2clr} are incorrect as the minus signs occurring in the numerators of the $q$-products of the generating functions are considered as plus. For example, for $n=6$,  $F_0(6)=24$ with the relevant partitions given by 
\begin{align*}
	&6_b,\\
	&5_b+1_b,\quad 5_r+1_r,\\
	&4_b+2_b,\\
	&4_b+1_b+1_b,\quad 4_b+1_r+1_r,\\
	&3_b+3_b,\quad 3_r+3_r,\\
	&3_b+2_b+1_b,\quad 3_r+2_b+1_r,\\
	&3_b+1_b+1_b+1_b,\quad 3_b+1_b+1_r+1_r,\\
	&3_r+1_b+1_b+1_r,\quad 3_r+1_r+1_r+1_r,\\
	&2_b+2_b+2_b,\\
	&2_b+2_b+1_b+1_b,\quad 2_b+2_b+1_r+1_r,\\
	&2_b+1_b+1_b+1_b+1_b,\quad
	2_b+1_b+1_b+1_r+1_r,\quad
	2_b+1_r+1_r+1_r+1_r,\\
	&1_b+1_b+1_b+1_b+1_b+1_b,\\
	&1_b+1_b+1_b+1_b+1_r+1_r,\\
	&1_b+1_b+1_r+1_r+1_r+1_r,\\
	&1_r+1_r+1_r+1_r+1_r+1_r.
\end{align*}
 and $F_1(6)=16$ with the relevant partitions given by 
\begin{align*}
	\hspace{-2.1cm}&5_b+1_r,\quad 5_r+1_b,\\
	\hspace{-2.1cm}	&4_b+1_b+1_r,\\
	\hspace{-2.1cm}	&3_b+3_r,\\
	\hspace{-2.1cm}	&3_b+2_b+1_r,\quad 3_r+2_b+1_b,\\
	\hspace{-2.1cm}	&3_b+1_b+1_b+1_r,\quad 3_b+1_r+1_r+1_r,\\
	\hspace{-2.1cm}	&3_r+1_b+1_b+1_b,\quad 3_r+1_b+1_r+1_r, \\
	\hspace{-2.1cm}	&2_b+2_b+1_b+1_r,\\
	\hspace{-2.1cm}	&2_b+1_b+1_b+1_b+1_r,\quad
	2_b+1_b+1_r+1_r+1_r,\\
	\hspace{-2.1cm}	&1_b+1_b+1_b+1_b+1_b+1_r,\\
	\hspace{-2.1cm}	&1_b+1_b+1_b+1_r+1_r+1_r,\\
	\hspace{-2.1cm}	&1_b+1_r+1_r+1_r+1_r+1_r,
\end{align*}
which satisfy $F_0(6)+F_1(6)=24+16=40=F(6)$.\\

\noindent$(ii)$ The partition interpretations of \( F_0(n) \) and \( F_1(n) \) given in \cite[p.2, Corollary 1.4 and 1.6]{2clr} are incorrect, whereas the interpretations given in \cite[Corollary 1 and 2]{arx} are correct, respectively.
\end{remark}

Recall that  an overpartition of a positive integer $n$ is a partition of $n$ wherein  the first occurrence (equivalently, the final occurrence) of a part may be overlined.  If $\overline{p}(n)$ denotes the number of overpartitions of $n$, then its generating function \cite{cor} is given by 
\begin{equation}\label{ovp}
	\sum_{n=0}^{\infty}\overline{p}(n)q^n=\dfrac{(-q;q)_\infty}{(q;q)_\infty}
\end{equation}
Andrews and Bachraoui \cite[p.4, Theorem 2.1]{2clr}  proved that
\begin{equation}\label{f0p}
   F_0(n)=\dfrac{\overline{p}(n)+\overline{p}(\frac{n}{2})}{2},  
\end{equation}
\begin{equation}\label{f1p}
   F_1(n)=\dfrac{\overline{p}(n)-\overline{p}(\frac{n}{2})}{2},  
\end{equation}
where $\overline p(\frac{n}{2})=0,$ whenever $n$ is odd.
Employing \eqref{ovp} in \eqref{f0p} and \eqref{f1p} and simplifying, we obtain
\begin{equation}\label{t3h}
\sum_{n\geq 0}F_0(n)q^n
=
\dfrac{1}{2}\left(\dfrac{f_2}{f_1^2}
+
\dfrac{f_4}{f_2^2}\right),
\end{equation}
and 
\begin{equation}\label{f1t3h} 
   \sum_{n\geq 0}F_1(n)q^n
=
\dfrac{1}{2}\left(\dfrac{f_2}{f_1^2}
-
\dfrac{f_4}{f_2^2}\right),
\end{equation}respectively.

The layout of this paper is as follows. In Sect. 2, we state  preliminaries that are useful in proving our results. In Sect. 3, we present the analytical proofs of $q$-series identities \eqref{OP1} and \eqref{2op5.6}. In Sect. 4, we prove congruences modulo $2$, $4$, and $8$ for $F_0(n)$ and $F_1(n)$.

\section{Preliminaries}
Ramanujan's general theta function $f(\alpha, \beta)$ \cite[p. 34, (18.1)]{BBC} is defined by
$$
   f(\alpha, \beta) = \sum_{m=-\infty}^{\infty} \alpha^{m(m+1)/2} \beta^{m(m-1)/2}, \qquad |\alpha\beta| < 1.
$$
Three important special cases for $f(\alpha, \beta)$ are the functions $\varphi(q), \psi(q)$, and $f(-q)$ \cite[p. 35, Entry 18]{BBC}, which are defined as
\begin{equation}\label{phi}
    \varphi(q) := f(q, q) = \sum_{m=-\infty}^{\infty} q^{m^2} = \dfrac{f_2^5}{f_1^2f_4^2},
\end{equation}
\begin{equation}\label{psi}
    \psi(q) := f(q, q^3) = \sum_{m=0}^{\infty} q^{m(m+1)/2} = \dfrac{f_2^2}{f_1},
\end{equation}
and 
\begin{equation}\label{f}
    f(-q) := f(-q, -q^2) = \sum_{m=-\infty}^{\infty} (-1)^m q^{m(3m-1)/2} = f_1.
\end{equation}
The product representations of the  special cases \eqref{phi}-\eqref{f} of $f(\alpha, \beta)$ are the consequences of Jacobi's triple product \cite[ p. 35, Entry 19]{BBC} given by
$$f(\alpha, \beta) = (-\alpha; \alpha\beta)_\infty(-\beta; \alpha\beta)_\infty(\alpha\beta; \alpha\beta)_\infty.$$
The Jacobi's triple product identity can also be stated as \cite{gas}
$$
(q,-xq,-1/x;q)_\infty
=
\sum_{n=-\infty}^{\infty}x^nq^{n(n+1)/2}.
$$
Let $\sigma(n)$ be the sum-of-divisors function, defined by 
$$
\sigma(n)=\sum_{d\mid n}d.  
$$
Then the Lambert series for  $\sigma(n)$ is given by 
	$$
	\sum_{n=1}^\infty \dfrac{n q^{n}}{1 - q^{n}}=\sum_{n=1}^\infty \sigma(n)q^n.	
    $$
Also, if $r_2(n)$ denotes the number of representations of a nonnegative integer $n$ as the sum of two squares of any integers, then
\begin{equation}\label{gp}
	\sum_{n=0}^\infty r_2(n)q^n=\varphi^2(q),
	\end{equation}
    where $r_2(0)=1.$
Also, 
from \cite[Eq. (3.2.1), (3.2.2)]{BBC}, we note that
\begin{equation}\label{r_2}
    r_2(n)=4\sum_{\substack{d|n\\d ~odd}}(-1)^{(d-1)/2}=4(d_{1,4}(n)-d_{3,4}(n)),
\end{equation} where, here and throughout the paper, $d_{k,j}(n)$ denotes the number of positive divisors of $n$ such that $d\equiv j\pmod k$.

\begin{lemma}
For every integer $n\geq0,$ we have 
\begin{equation}\label{r_2.4n+1}
r_2(4n+1)=4\sum_{d\mid 4n+1}\left(\dfrac{-1}{d}\right).
\end{equation}
\end{lemma}
\begin{proof}
From \eqref{r_2}, for $N$ being a positive integer, we have
\begin{equation}\label{A4}
r_2(N)=4(d_{1,4}(N)-d_{3,4}(N)).
\end{equation}
If  $N$ is odd,  every divisor $d$ of $N$ is odd. So
$$\left(\dfrac{-1}{d}\right)=
\begin{cases}
1, & \text{if } d\equiv1\pmod 4,\\
-1, & \text{if } d\equiv3\pmod 4, 
\end{cases}$$ where $\left(\dfrac{.}{.}\right)$ denotes the Jacobi's symbol. Therefore, 
\begin{equation}\label{A6}
\sum_{d\mid N}\left(\dfrac{-1}{d}\right)
=
d_{1,4}(N)-d_{3,4}(N).
\end{equation}
Employing \eqref{A6} in \eqref{A4}, we obtain
\begin{equation}\label{r1}
r_2(N)=4\sum_{d\mid N}\left(\dfrac{-1}{d}\right).
\end{equation}
Setting $N=4n+1$ in \eqref{r1}, we arrive at the desired result. 
\end{proof}

\begin{lemma}
For every integer $n\geq0,$ 
let
\begin{equation}
r_{1,2}(n)=\#\{(u,v)\in\mathbb Z^2:u^2+2v^2=n\}.\notag
\end{equation}Then, we have 
\begin{equation}\label{B3}
r_{1,2}(4n+1)=2\sum_{d\mid 4n+1}\left(\dfrac{-2}{d}\right).
\end{equation}
\end{lemma}
\begin{proof}
From \cite[Theorem 3.7.3]{bc}, for $N$ being a positive integer, we have
\begin{equation}\label{B4}
r_{1,2}(N)=2\left(d_{1, 8}(N)+d_{3, 8}(N)-d_{5, 8}(N)-d_{7, 8}(N)\right),
\end{equation} If $N$ is odd, every divisor $d$ of $N$ is odd. Thus,
\begin{equation}
\left(\dfrac{-2}{d}\right)=
\begin{cases}
1, & \text{if } d\equiv1,3\pmod 8,\\[3pt]
-1, & \text{if } d\equiv5,7\pmod 8.
\end{cases}\notag
\end{equation}
Therefore, 
\begin{equation}\label{B6}
\sum_{d\mid N}\left(\dfrac{-2}{d}\right)
=
d_{1,8}(N)+d_{3,8}(N)-d_{5,8}(N)-d_{7,8}(N).
\end{equation}
Employing \eqref{B6} in \eqref{B4}, we obtain
\begin{equation}\label{r2}
r_{1,2}(N)=2\sum_{d\mid N}\left(\dfrac{-2}{d}\right).
\end{equation}
Setting $N=4n+1$ in \eqref{r2}, we complete the proof.
\end{proof}

\begin{lemma}\cite{BBC} We have
\begin{equation}\label{P5}
\varphi(q)=\varphi(q^4)+2q\psi(q^8).
\end{equation}    
\end{lemma}

\begin{lemma} We have
   \begin{align}
&\label{P6}\dfrac{1}{f_1^2}
=
\dfrac{f_8^5}{f_2^5f_{16}^2}
+
2q\dfrac{f_4^2f_{16}^2}{f_2^5f_8},\\
&\label{P7}\dfrac{1}{f_1^4}
=
\dfrac{f_4^{14}}{f_2^{14}f_8^4}
+
4q\dfrac{f_4^2f_8^4}{f_2^{10}}.
\end{align}
\end{lemma}
\begin{proof}
    The above two identities are the $2$-dissection of $\varphi(q)$ and $\varphi(q^2)$ respectively (see \cite[Eqs. (1.9.4) and (1.10.1)]{md}).
\end{proof}
The next lemma is an easy consequence of the binomial theorem.
\begin{lemma}For any positive integer $t,$ we have 
\begin{equation}\label{2t}
    f_1^{2^t}\equiv f_{2}^{2^{t-1}}\pmod{2^t}. 
\end{equation}   
\end{lemma}

\section{Analytic proofs of $q$-series identities \eqref{OP1} and \eqref{2op5.6}}
In this section, we provide analytic proofs of the \(q\)-series identities \eqref{OP1} and \eqref{2op5.6} connected with the partition functions  \( F(n) \) and \( H(n) \), which appeared as open problems in  \cite[p. 10-11, (5.3) and (5.6)] {2clr}, respectively:\\

\noindent{\bf \textit{Proof of the identity \eqref{OP1}}:}
For any  integer $n\ge 0$, let
\begin{equation}\label{1ope1}
A_n=\dfrac{1}{(q^2;q^2)_n(q;q^2)_{n+1}^2}
\end{equation}
and $A_{-1}=0$. Then for any integer $n\geq1$,
\begin{align}
   A_n-A_{n-1}&=\dfrac{1}{(q^2;q^2)_n(q;q^2)_n^2(1-q^{2n+1})^2}-\dfrac{1}{(q^2;q^2)_{n-1}(q;q^2)_n^2}\notag\\
   &=\dfrac{1}{(q^2;q^2)_n(q;q^2)_n^2(1-q^{2n+1})^2}-\dfrac{1-q^{2n}}{(q^2;q^2)_n(q;q^2)_n^2}\notag\\
   &=\dfrac{1}{(q^2;q^2)_n(q;q^2)_n^2}
\left(
\dfrac{1}{(1-q^{2n+1})^2}-(1-q^{2n})
\right)\notag\\
&=\dfrac{1}{(q^2;q^2)_n(q;q^2)_n^2}\left(\dfrac{1-(1-q^{2n})(1-q^{2n+1})^2}{(1-q^{2n+1})^2}\right)\notag\\
&=\dfrac{1}{(q^2;q^2)_n(q;q^2)_n^2}\left(\dfrac{q^{2n}\left((1-q^{2n+1})^2+q(1-q^{2n+1})+q\right)}{(1-q^{2n+1})^2}\right)\notag\\
&=\dfrac{q^{2n}}{(q^2;q^2)_n(q;q^2)_n^2}
\left(
1+\dfrac{q}{1-q^{2n+1}}
+\dfrac{q}{(1-q^{2n+1})^2}
\right), \notag
\end{align}
which implies
\begin{equation}\label{1ope3}\
\sum_{n=0}^{N}
\dfrac{q^{2n}}{(q^2;q^2)_n(q;q^2)_n^2}
\left(
1+\dfrac{q}{1-q^{2n+1}}
+\dfrac{q}{(1-q^{2n+1})^2}
\right)
=\sum_{n=0}^{N}(A_n-A_{n-1})
=A_N-A_{-1}=A_N.
\end{equation}
Employing  \eqref{1ope1} in \eqref{1ope3} and then taking limit $N\rightarrow\infty$, we obtain
\begin{equation}\label{1ope4}\sum_{n=0}^{\infty}
\dfrac{q^{2n}}{(q^2;q^2)_n(q;q^2)_n^2}
\left(
1+\dfrac{q}{1-q^{2n+1}}
+\dfrac{q}{(1-q^{2n+1})^2}
\right)
=\lim_{N\to\infty} A_N=\dfrac{1}{(q^2;q^2)_\infty(q;q^2)_\infty^2}.
\end{equation}
Again,   for any integer $n\ge 0$, let
\begin{equation}\label{bn}B_n=
\dfrac{1}{(q^{2n+2};q^2)_\infty(q^{2n+1};q^2)_\infty^2}.\end{equation}
Now
\begin{align}
B_n-B_{n+1}
&=
\dfrac{1}{(q^{2n+2};q^2)_\infty(q^{2n+1};q^2)_\infty^2}
-
\dfrac{1}{(q^{2n+4};q^2)_\infty(q^{2n+3};q^2)_\infty^2}
\notag\\
&=
\dfrac{1}{(q^{2n+2};q^2)_\infty(q^{2n+3};q^2)_\infty^2}
\left(
\dfrac{1}{(1-q^{2n+1})^2}
-(1-q^{2n+2})
\right)
\notag\\
&=
\dfrac{1}{(q^{2n+2};q^2)_\infty(q^{2n+3};q^2)_\infty^2}
\left(
\dfrac{1-(1-q^{2n+2})(1-q^{2n+1})^2}{(1-q^{2n+1})^2}
\right)
\notag\\
&=
\dfrac{1}{(q^{2n+2};q^2)_\infty(q^{2n+3};q^2)_\infty^2}
\left(
\dfrac{
q^{2n+1}
\left(
q(1-q^{2n+1})^2+(1-q^{2n+1})+1
\right)
}
{(1-q^{2n+1})^2}
\right)
\notag\\
&=
\dfrac{q^{2n+1}}{(q^{2n+2};q^2)_\infty(q^{2n+3};q^2)_\infty^2}
\left(
q+\dfrac{1}{1-q^{2n+1}}
+\dfrac{1}{(1-q^{2n+1})^2}
\right), \notag
\end{align}
which  implies
\begin{equation}\label{2ope5}
\sum_{n=0}^{N}
\dfrac{q^{2n+1}}{(q^{2n+2};q^2)_\infty(q^{2n+3};q^2)_\infty^2}
\left(
q+\dfrac{1}{1-q^{2n+1}}
+\dfrac{1}{(1-q^{2n+1})^2}
\right)
=
\sum_{n=0}^{N}(B_n-B_{n+1})
=
B_0-B_{N+1}.
\end{equation}
Employing \eqref{bn} in \eqref{2ope5}  and then taking limit $N\rightarrow\infty$, we obtain
\begin{equation}\label{oa1}
1+\sum_{n=0}^{\infty}
\dfrac{q^{2n+1}}{(q^{2n+2};q^2)_\infty(q^{2n+3};q^2)_\infty^2}
\left(
q+\dfrac{1}{1-q^{2n+1}}
+\dfrac{1}{(1-q^{2n+1})^2}
\right)
=
\dfrac{1}{(q^2;q^2)_\infty(q;q^2)_\infty^2}.
\end{equation}
Combining \eqref{oa1} and \eqref{1ope4}, we complete the proof of  \eqref{OP1}.\\

\noindent {\bf \textit{Proof of the identity \eqref{2op5.6}}:}
For any integer $n\geq0$, let
\begin{equation}\label{1op56e1}
C_n=(-q^2;q^2)_{n+1}(-q;q^2)_{n+1}^2,
\end{equation}
and $C_{-1}=1.$ Then for any integer $n\geq1,$  we have
\begin{align}
C_n-C_{n-1}
&=
(-q^2;q^2)_{n+1}(-q;q^2)_{n+1}^2
-
(-q^2;q^2)_n(-q;q^2)_n^2
\notag\\
&=
(-q^2;q^2)_n(-q;q^2)_{n+1}^2(1+q^{2n+2})
-
\dfrac{(-q^2;q^2)_n(-q;q^2)_{n+1}^2}{(1+q^{2n+1})^2}
\notag\\
&=
(-q^2;q^2)_n(-q;q^2)_{n+1}^2
\left(
1+q^{2n+2}
-\dfrac{1}{(1+q^{2n+1})^2}
\right)
\notag\\
&=
(-q^2;q^2)_n(-q;q^2)_{n+1}^2
\left(
\dfrac{(1+q^{2n+2})(1+q^{2n+1})^2-1}{(1+q^{2n+1})^2}
\right)
\notag\\
&=
(-q^2;q^2)_n(-q;q^2)_{n+1}^2
\left(
\dfrac{
q^{2n+1}
\left(
q(1+q^{2n+1})^2+(1+q^{2n+1})+1
\right)
}
{(1+q^{2n+1})^2}
\right)
\notag\\
&=
q^{2n+1}(-q^2;q^2)_n(-q;q^2)_{n+1}^2
\left(
q+\dfrac{1}{1+q^{2n+1}}
+\dfrac{1}{(1+q^{2n+1})^2}
\right)\notag
\end{align}
which implies
\begin{equation}\label{1op56e3}
\sum_{n=0}^{N}
q^{2n+1}(-q^2;q^2)_n(-q;q^2)_{n+1}^2
\left(
q+\dfrac{1}{1+q^{2n+1}}
+\dfrac{1}{(1+q^{2n+1})^2}
\right)
=
\sum_{n=0}^{N}(C_n-C_{n-1})
=
C_N-C_{-1}.
\end{equation}
Employing \eqref{1op56e1} in \eqref{1op56e3}  and then taking limit $N\rightarrow \infty$, we obtain 
$$\hspace{-5cm}\sum_{n=0}^{\infty}
q^{2n+1}(-q^2;q^2)_n(-q;q^2)_{n+1}^2
\left(
q+\dfrac{1}{1+q^{2n+1}}
+\dfrac{1}{(1+q^{2n+1})^2}
\right)
$$\begin{equation}\label{OP2E1}=
\lim_{N\to\infty}(C_N-C_{-1})
=
(-q^2;q^2)_\infty(-q;q^2)_\infty^2-1.
\end{equation}
Again, for any integer $n\ge0$,  let
\begin{equation}\label{2op56e1}
D_n=(-q^{2n+1};q^2)_\infty^2(-q^{2n+2};q^2)_\infty.
\end{equation}
Now, 
\begin{align}
D_n-D_{n+1}
&=
(-q^{2n+1};q^2)_\infty^2(-q^{2n+2};q^2)_\infty
-
(-q^{2n+3};q^2)_\infty^2(-q^{2n+4};q^2)_\infty
\notag\\
&=
(1+q^{2n+1})
(-q^{2n+1},-q^{2n+2},-q^{2n+3};q^2)_\infty
-
\dfrac{(-q^{2n+1},-q^{2n+2},-q^{2n+3};q^2)_\infty}
{(1+q^{2n+1})(1+q^{2n+2})}
\notag\\
&=
(-q^{2n+1},-q^{2n+2},-q^{2n+3};q^2)_\infty
\left(
1+q^{2n+1}
-
\dfrac{1}{(1+q^{2n+1})(1+q^{2n+2})}
\right)
\notag\\
&=
(-q^{2n+1},-q^{2n+2},-q^{2n+3};q^2)_\infty
\left(
\dfrac{
(1+q^{2n+1})^2(1+q^{2n+2})-1
}
{(1+q^{2n+1})(1+q^{2n+2})}
\right)
\notag\\
&=
(-q^{2n+1},-q^{2n+2},-q^{2n+3};q^2)_\infty
\notag\\
&\qquad \times
\left(
\dfrac{
q^{2n+1}
\left(
(1+q^{2n+1})(1+q^{2n+2})
+(1+q^{2n+2})
+q(1+q^{2n+1})
\right)
}
{(1+q^{2n+1})(1+q^{2n+2})}
\right)
\notag\\
&=
q^{2n+1}(-q^{2n+1},-q^{2n+2},-q^{2n+3};q^2)_\infty
\left(
1+\dfrac{1}{1+q^{2n+1}}
+\dfrac{q}{(1+q^{2n+1})(1+q^{2n+2})}
\right)\notag
\end{align}
which implies
\begin{align}
&\sum_{n=0}^{N}
q^{2n+1}(-q^{2n+1},-q^{2n+2},-q^{2n+3};q^2)_\infty
\left(
1+\dfrac{1}{1+q^{2n+1}}
+\dfrac{q}{(1+q^{2n+1})(1+q^{2n+2})}
\right)
\notag\\
&\label{2op56e5}=
\sum_{n=0}^{N}(D_n-D_{n+1})
=
D_0-D_{N+1}.
\end{align}
Employing \eqref{2op56e1} in \eqref{2op56e5} and taking limit $N\rightarrow\infty$, we obtain
\begin{align}
&\sum_{n=0}^{\infty}
q^{2n+1}(-q^{2n+1},-q^{2n+2},-q^{2n+3};q^2)_\infty
\left(
1+\dfrac{1}{1+q^{2n+1}}
+\dfrac{q}{(1+q^{2n+1})(1+q^{2n+2})}
\right)
\notag\\
&=
\lim_{N\to\infty}(D_0-D_{N+1})
\notag\\
&\label{OPE2}=
D_0-1=
(-q^2;q^2)_\infty(-q;q^2)_\infty^2-1.
\end{align}
Combining \eqref{OP2E1} and \eqref{OPE2}, we complete the proof of the  identity \eqref{2op5.6}.

\section{Congruences for $F_0(n)$ and $F_1(n)$} 
This section is devoted to proving some congruences modulo \( 2\), \( 4\),  and \( 8 \)  for the partition functions  \( F_0(n) \) and \( F_1(n) \).

\begin{theorem}\label{TM2}For any integer $n\ge1$ and $i\in\{0, 1\}$, we have
 \begin{equation}\label{Tm2}
F_{i}(n)\equiv
\begin{cases}
1 \pmod 2, & \text{if } n \text{ is a perfect square or twice of a perfect square},\\[4pt]
0 \pmod 2, & \text{otherwise}.
\end{cases}    
\end{equation}
\end{theorem}
\begin{proof}We have
\begin{align}
\dfrac{f_2}{f_1^2}=\dfrac{(q^2;q^2)_\infty}{(q;q)_\infty^2}
&=
\varphi(q)\,
\dfrac{(q^4;q^4)_\infty^2}{(q^2;q^2)_\infty^4} \notag \\
&=
\varphi(q)
\prod_{m=1}^\infty
\left(\dfrac{1+q^{2m}}{1-q^{2m}}\right)^2\notag\\
&=\varphi(q)\prod_{m=1}^\infty\left(1+\dfrac{4q^{2m}}{(1-q^{2m})^2}\right)\notag\\
&\equiv \varphi(q) \pmod 4\notag\\
&\label{m4f0}\equiv
1+2\sum_{r=1}^\infty q^{r^2}
\pmod 4.
\end{align}
Replacing $q$ by $q^2$ in \eqref{m4f0}, we obtain 
\begin{equation}\label{m4f01}
    \dfrac{f_4}{f_2^2}= \dfrac{(q^4;q^4)_\infty}{(q^2;q^2)_\infty^2}\equiv1+2\sum_{r=1}^{\infty}q^{2r^2}
\pmod 4.
\end{equation}
Employing \eqref{m4f0} and \eqref{m4f01} in \eqref{t3h}, we obtain
\begin{equation}\label{m4f02}
    2\sum_{n= 0}^\infty F_0(n)q^n\equiv2+2\sum_{r=1}^\infty q^{r^2}+2\sum_{r=1}^\infty q^{2r^2}\pmod 4, 
\end{equation}
which is equivalent to
\begin{equation}\label{m4f03}
    \sum_{n=0}^\infty F_0(n)q^n \equiv1+\sum_{r=1}^\infty q^{r^2}+\sum_{r=1}^\infty q^{2r^2}\pmod 2. 
\end{equation}Equating the coefficients of $q^n$ on both sides of \eqref{m4f03}, we complete the proof of the case $i=0$.

Again, employing \eqref{m4f0} and \eqref{m4f01} in \eqref{f1t3h}, we obtain
\begin{equation}\label{m4f04}
    \sum_{n\geq 0}F_1(n)q^n \equiv\sum_{r\geq1}q^{r^2}+\sum_{r\geq1}q^{2r^2}\pmod 2. 
\end{equation}
Equating the coefficients of $q^n$ on both sides of \eqref{m4f04}, we complete the proof of the case $i=1$.
\end{proof}

\begin{corollary}\label{cm2}
Let $M$ and $r$ be integers such that  $M\geq1$ and $0\leq r<M$. Let
\begin{equation}
\mathcal Q_M
=
\{x^2\pmod M:x\in\mathbb Z\}
\cup
\{2x^2\pmod M:x\in\mathbb Z\}.\notag
\end{equation}
Then, for $i\in\{0,1\}$, we have 
\begin{equation}\label{g2}
F_i(Mn+r)\equiv0\pmod 2\text{ for all $n\geq0$ }
\iff
r\notin \mathcal Q_M.\notag
\end{equation}
\end{corollary}
\begin{proof}Suppose 
$
r\notin \mathcal Q_M
$ and  for some integer $n\ge0$, 
\begin{equation}\label{g5}
F_i(Mn+r)\equiv1\pmod 2
\end{equation} Then by Theorem \ref{TM2} and \eqref{g5}, for some integer $x\geq0$,
\begin{equation}
Mn+r=x^2
\quad \text{or} \quad
Mn+r=2x^2\notag
\end{equation}
 which implies $r\in Q_M,$ which is a contradiction.  \\
Conversely, assume that for any integer $n\ge 0$, 
\begin{equation}\label{g9}
F_i(Mn+r)\equiv0\pmod 2
\end{equation} holds.
Suppose $
r\in\mathcal Q_M.
$
Then for some integer $x$, either
$
r\equiv x^2\pmod M
\text{ or }
r\equiv2x^2\pmod M.
$
If $r\equiv x^2\pmod M$,  choose an integer $t$ sufficiently large such that
$
X=x+tM. $
Then
$
X^2\equiv x^2\equiv r\pmod M.$
Also, for sufficiently large $t$, $X^2\geq r$. Therefore, 
$
N=\dfrac{X^2-r}{M}$
is a nonnegative integer and 
$
MN+r=X^2.\notag
$
So by Theorem \ref{TM2}, we have
\begin{equation}
F_i(MN+r)=F_i(X^2)\equiv1\pmod 2,\notag
\end{equation}
which is a contradiction to  \eqref{g9}.\\
Again, if  $r\equiv2x^2\pmod M,$  choose an integer $t$ sufficiently large such that 
$
X=x+tM.
$
\begin{equation}
2X^2\equiv 2x^2\equiv r\pmod M.\notag
\end{equation}
Therefore, $M$ divides $2X^2-r$  and also for sufficiently large $t$, $2X^2\geq r$. So
$
N=\dfrac{2X^2-r}{M}$
is a nonnegative integer and $
MN+r=2X^2.$
Then Theorem \ref{TM2} implies
\begin{equation}
F_i(MN+r)=F_i(2X^2)\equiv1\pmod 2,\notag
\end{equation}
which is a contradiction to \eqref{g9}. Therefore, 
$
r\notin\mathcal Q_M.
$
Hence, the proof is complete.
\end{proof}

\begin{corollary}
Let $p$ be an odd prime. Let $\alpha$ and $s$ be any nonnegative integers. If $B$ is a positive integer such that $p\nmid B,$ then for $i\in\{0,1\}$, 
\begin{equation}
F_i\left(2^\alpha p^{2s+1}B\right)\equiv0\pmod{2}.\notag
\end{equation}
\end{corollary}
\begin{theorem}
For any integer $n\geq 1$, we have
\begin{equation}\label{L1}
nF_0(n)\equiv \sum_{k=1}^{n}\sigma(k)F_0(n-k) \pmod 2,
\end{equation}
and
\begin{equation}\label{L2}
nF_1(n)\equiv F_1(n)+\sum_{k=1}^{n}\sigma(k)F_1(n-k) \pmod 2.
\end{equation}
\end{theorem}

\begin{proof}
Taking the logarithm on both sides of \eqref{I3}, we obtain
\begin{align}
\log\left(\sum_{n=0}^{\infty}F_0(n)q^n\right)
=&
\log(q^{16};q^{16})_{\infty}
+\log(-q^6;q^{16})_{\infty}
+\log(-q^{10};q^{16})_{\infty} \notag\\
&\label{last1}-\log(q;q^2)_{\infty}
-2\log(q^2;q^2)_{\infty}.
\end{align}
Differentiating \eqref{last1} with respect to $q$, we obtain
\begin{align}
\dfrac{\sum_{n=0}^{\infty}nF_0(n)q^{n-1}}{\sum_{n=0}^{\infty}F_0(n)q^n}
=&
-16\sum_{m=1}^{\infty}\dfrac{mq^{16m-1}}{1-q^{16m}}
+\sum_{m=0}^{\infty}\dfrac{(16m+6)q^{16m+5}}{1+q^{16m+6}} 
+\sum_{m=0}^{\infty}\dfrac{(16m+10)q^{16m+9}}{1+q^{16m+10}}\notag\\
&\label{last2}\hspace{3.4cm}+\sum_{m=0}^{\infty}\dfrac{(2m+1)q^{2m}}{1-q^{2m+1}} 
+4\sum_{m=1}^{\infty}\dfrac{mq^{2m-1}}{1-q^{2m}}.
\end{align}
Multiplying  \eqref{last2} by $q$, we obtain
\begin{align}
\dfrac{\sum_{n=0}^{\infty}nF_0(n)q^{n}}{\sum_{n=0}^{\infty}F_0(n)q^n}
=&
-16\sum_{m=1}^{\infty}\dfrac{mq^{16m}}{1-q^{16m}}
+\sum_{m=0}^{\infty}\dfrac{(16m+6)q^{16m+6}}{1+q^{16m+6}} 
+\sum_{m=0}^{\infty}\dfrac{(16m+10)q^{16m+10}}{1+q^{16m+10}}\notag\\
&\hspace{3.4cm}+\sum_{m=0}^{\infty}\dfrac{(2m+1)q^{2m+1}}{1-q^{2m+1}} 
+4\sum_{m=1}^{\infty}\dfrac{mq^{2m}}{1-q^{2m}}\notag\\
\equiv&
\sum_{m=0}^{\infty}\dfrac{(2m+1)q^{2m+1}}{1-q^{2m+1}}
\pmod 2\notag\\
\equiv&
\sum_{n=1}^{\infty}\left(\sum_{\substack{d\mid n\\2\nmid d}}d\right)q^n\pmod2\notag\\
\equiv&\label{last3}\sum_{n=1}^{\infty}\sigma(n)q^n\pmod2.
\end{align}
Now  \eqref{L1} follows easily from \eqref{last3}.\\

Again,  taking the logarithm on both sides of \eqref{I4}, we obtain
\begin{align}
\log \left(\sum_{n=0}^{\infty}F_1(n)q^n\right)
={}&
\log q
+\log(q^{16};q^{16})_{\infty}
+\log(-q^2;q^{16})_{\infty}
+\log(-q^{14};q^{16})_{\infty} \notag\\
&\label{last4}-\log(q;q^2)_{\infty}
-2\log(q^2;q^2)_{\infty}.
\end{align}
Differentiating \eqref{last4} with respect to $q$, we obtain
\begin{align}
\dfrac{\sum_{n=0}^{\infty}nF_1(n)q^{n-1}}{\sum_{n=0}^{\infty}F_1(n)q^n}
=&
\dfrac{1}{q}
-16\sum_{m=1}^{\infty}\dfrac{mq^{16m-1}}{1-q^{16m}}
+\sum_{m=0}^{\infty}\dfrac{(16m+2)q^{16m+1}}{1+q^{16m+2}} 
+\sum_{m=0}^{\infty}\dfrac{(16m+14)q^{16m+13}}{1+q^{16m+14}}\notag\\
&\label{last5}\hspace{4cm}+\sum_{m=0}^{\infty}\dfrac{(2m+1)q^{2m}}{1-q^{2m+1}} 
+4\sum_{m=1}^{\infty}\dfrac{mq^{2m-1}}{1-q^{2m}}.
\end{align}
Multiplying both sides of \eqref{last5} by $q$, we obtain
\begin{align}
\dfrac{\sum_{n=0}^{\infty}nF_1(n)q^{n}}{\sum_{n=0}^{\infty}F_1(n)q^n}
=&
1
-16\sum_{m=1}^{\infty}\dfrac{mq^{16m}}{1-q^{16m}}
+\sum_{m=0}^{\infty}\dfrac{(16m+2)q^{16m+2}}{1+q^{16m+2}} 
+\sum_{m=0}^{\infty}\dfrac{(16m+14)q^{16m+14}}{1+q^{16m+14}}\notag\\
&\hspace{4cm}+\sum_{m=0}^{\infty}\dfrac{(2m+1)q^{2m+1}}{1-q^{2m+1}} 
+4\sum_{m=1}^{\infty}\dfrac{mq^{2m}}{1-q^{2m}}\notag\\
\equiv&\label{last6}1+\sum_{n=1}^{\infty}\sigma(n)q^n
\pmod 2.
\end{align}
Now \eqref{L2} follows easily from \eqref{last6}.
\end{proof}
\begin{theorem}
For any positive integer $n,$ we have 
\begin{equation}\label{t3r}
F_0(n)\equiv
\begin{cases}
1 \pmod 4, & \text{if $n=x^2$ for some positive integer $x$},\\[4pt]
3 \pmod 4, & \text{if $n=2x^2$ for some odd positive integer $x$},\\[4pt]
1 \pmod 4, & \text{if $n=2x^2$ for some even positive integer $x$},\\[4pt]
0 \pmod 4, & \text{if $n\neq x^2,\ 2x^2$ for some positive integer $x,$}
\end{cases}
\end{equation}
and 
\begin{equation}\label{A35}
F_1(n)\equiv
\begin{cases}
1 \pmod 4, & \text{if  $n=x^2$ for some positive integer $x$}\\[4pt]
1 \pmod 4, & \text{if $n=2x^2$ for some odd positive integer $x$},\\[4pt]
3 \pmod 4, & \text{if $n=2x^2$ for some even positive integer $x$},\\[4pt]
0 \pmod 4, & \text{if $n\neq x^2,\ 2x^2$ for some positive integer $x.$}
\end{cases}
\end{equation}
\end{theorem}
\begin{proof}
Using \eqref{2t} and \eqref{phi}, we obtain 
\begin{align}
\dfrac{f_2}{f_1^2}
=
\dfrac{f_2f_1^6}{f_1^8}&\equiv
\dfrac{f_1^6}{f_2^3}
=
\varphi(-q)^3 =\left(
1+2\sum_{r=1}^\infty(-1)^r q^{r^2}
\right)^3 \pmod 8\notag\\
&\label{t3l}\equiv
1
+
6\sum_{r=1}^\infty(-1)^r q^{r^2}
+
4\left(
\sum_{r=1}^\infty(-1)^r q^{r^2}
\right)^2\pmod 8 .
\end{align}
Also, we note that
\begin{equation}\label{t3m}
\left(
\sum_{r=1}^\infty(-1)^r q^{r^2}
\right)^2
=
\sum_{a=1}^{\infty}\sum_{b=1}^\infty(-1)^{a+b}q^{a^2+b^2} 
\equiv
\sum_{a=1}^\infty\sum_{b=1}^\infty q^{a^2+b^2}
\equiv
\sum_{r=1}^\infty q^{2r^2}
\pmod 2.
\end{equation}
Employing  \eqref{t3m} in \eqref{t3l},  we obtain
\begin{equation}\label{t3n}
\dfrac{f_2}{f_1^2}
\equiv
1
+
6\sum_{r\geq 1}(-1)^r q^{r^2}
+
4\sum_{r\geq 1}q^{2r^2}
\pmod 8 .
\end{equation}
Replacing $q$ by $q^2$ in \eqref{t3n}, we obtain
\begin{equation}\label{t3o}
\dfrac{f_4}{f_2^2}
\equiv
1
+
6\sum_{r\geq 1}(-1)^r q^{2r^2}
+
4\sum_{r\geq 1}q^{4r^2}
\pmod 8 .
\end{equation}
Employing \eqref{t3n} and \eqref{t3o} in \eqref{t3h}, we obtain
\begin{equation}\label{t3p}
\sum_{n\geq 0}F_0(n)q^n
\equiv
1
+
3\sum_{r\geq 1}(-1)^r q^{r^2}
+
2\sum_{r\geq 1}q^{2r^2} 
+3\sum_{r\geq 1}(-1)^r q^{2r^2}
+2\sum_{r\geq 1}q^{4r^2}
\pmod 4 .
\end{equation}
Now comparing the coefficients of the like powers of $q$ in \eqref{t3p}, we arrive at \eqref{t3r}. Similarly, \eqref{A35} can be proved by employing \eqref{t3n} and \eqref{t3o} in \eqref{f1t3h}  and comparing the  coefficients of the like powers of $q$.
\end{proof}

The proof of  Corollary \ref{cm4} is identical to the proof of Corollary \ref{cm2},  so we give the statement only and omit proof.
\begin{corollary}\label{cm4}Let $M$ and $r$ be integers such that  $M\geq1$ and $0\leq r<M$. Define
\begin{equation}
\mathcal Q_M
=
\{x^2\pmod M:x\in\mathbb Z\}
\cup
\{2x^2\pmod M:x\in\mathbb Z\}.\notag
\end{equation}
Then, for $i\in\{0,1\}$, we have 
\begin{equation}
F_i(Mn+r)\equiv0\pmod 4
\text{ for all $n\geq0$ }\iff
r\notin \mathcal Q_M.\notag
\end{equation}
\end{corollary}
\begin{remark}\label{r4m2}\it
From  Corollary \ref{cm4}, for any integer $n\geq 0$, and $i\in\{0,1\}$, the  following congruences  for $F_i(n)$ can be easily obtained:
\begin{align}
&F_i(4n+3)\equiv0\pmod{4}, \notag\\
&F_i(6n+5)\equiv0\pmod{4}, \notag\\
&F_i(7n+3)\equiv F_i(7n+5)\equiv F_i(7n+6)\equiv0\pmod{4},\notag\\
&F_i(8n+3)\equiv F_i(8n+5)\equiv F_i(8n+6)
\equiv F_i(8n+7)\equiv0\pmod{4}, \notag\\
&F_i(9n+3)\equiv F_i(9n+6)\equiv0\pmod{4}, \notag\\
&F_i(10n+3)\equiv F_i(10n+7)\equiv0\pmod{4}, \notag\\
&F_i(12n+3)\equiv F_i(12n+5)\equiv F_i(12n+7)
\equiv F_i(12n+10)\equiv F_i(12n+11)\equiv0\pmod{4}, \notag\\
&F_i(14n+3)\equiv F_i(14n+5)\equiv F_i(14n+6)
\equiv F_i(14n+10) \equiv F_i(14n+12)\equiv F_i(14n+13)\equiv0\pmod{4}, \notag\\
&F_i(15n+7)\equiv F_i(15n+11)\equiv F_i(15n+13)
\equiv F_i(15n+14)\equiv0\pmod{4}, \mbox{and so on.}\notag
\end{align} 
\end{remark}
\begin{corollary}
Let $p$ be an odd prime. Let $\alpha$ and $s$ be any nonnegative integers.  If $B$ is a positive integer such that $p\nmid B$, then for $i\in\{0,1\}$
\begin{equation}
F_i\left(2^\alpha p^{2s+1}B\right)\equiv0\pmod 4.\notag
\end{equation}
\end{corollary}

\begin{theorem}
For any integer $n\geq1,$ we have 
\begin{equation}\label{e3}
F_0(2n)\equiv
\begin{cases}
3 \pmod 8, & \text{if } n=x^2 \text{ for some odd integer }x,\\[4pt]
1 \pmod 8, & \text{if } n=x^2 \text{ for some even integer }x,\\[4pt]
1 \pmod 8, & \text{if } n=2x^2 \text{ for some integer }x\geq0,\\[4pt]
0 \pmod 8, & \text{if } n\neq x^2,\ 2x^2
\text{ for every integer }x\geq0.
\end{cases}
\end{equation}
and 
\begin{equation}\label{A53}
F_1(2n)\equiv
\begin{cases}
1 \pmod 8, & \text{if } n=x^2 \text{ for some odd integer } x,\\
3 \pmod 8, & \text{if } n=x^2 \text{ for some even integer } x,\\
5 \pmod 8, & \text{if } n=2x^2 \text{ for some integer } x\geq1,\\
0 \pmod 8, & \text{if } n\neq x^2,\ 2x^2 \text{ for every integer } x\geq1.
\end{cases}
\end{equation}
\end{theorem}
\begin{proof}
Employing \eqref{P6} in \eqref{t3h}, we obtain
\begin{align}
\sum_{n=0}^{\infty}F_0(n)q^n
&=
\dfrac{1}{2}
\left(
\dfrac{f_8^5}{f_2^4f_{16}^2}
+
2q\dfrac{f_4^2f_{16}^2}{f_2^4f_8}
+
\dfrac{f_4}{f_2^2}
\right).
\label{A14f0}
\end{align}
Extracting the terms involving even powers of $q$ from both sides of \eqref{A14f0} and replacing $q^2$ by $q$, we obtain
\begin{equation}\label{A15f0}
\sum_{n=0}^{\infty}F_0(2n)q^n
=
\dfrac{1}{2}
\left(
\dfrac{f_4^5}{f_1^4f_8^2}
+
\dfrac{f_2}{f_1^2}
\right).
\end{equation}
Employing \eqref{P6} and \eqref{P7} in \eqref{A15f0}, we obtain 
\begin{align}
2\sum_{n=0}^{\infty}F_0(2n)q^n
&=
\dfrac{f_4^{19}}{f_2^{14}f_8^6}
+
\dfrac{f_8^5}{f_2^4f_{16}^2}
+
q\left(
4\dfrac{f_4^7f_8^2}{f_2^{10}}
+
2\dfrac{f_4^2f_{16}^2}{f_2^4f_8}
\right).
\label{e12}
\end{align}
Employing \eqref{2t} in \eqref{e12}, we obtain
\begin{equation}
2\sum_{n=0}^{\infty}F_0(2n)q^n
\equiv
\dfrac{f_4^5}{f_2^2f_8^2}
+
\dfrac{f_8^5}{f_4^2f_{16}^2}
+
6q\dfrac{f_{16}^2}{f_8}
\pmod {16}.
\label{e14}
\end{equation}
Employing \eqref{phi} and \eqref{psi} in \eqref{e14}, we obtain
\begin{equation}
2\sum_{n=0}^{\infty}F_0(2n)q^n
\equiv
\varphi(q^2)+\varphi(q^4)+6q\psi(q^8)
\pmod {16}.
\label{e16}
\end{equation}
Employing \eqref{P5} in \eqref{e16}, we obtain
\begin{equation}\label{e19}
2\sum_{n=0}^{\infty}F_0(2n)q^n
\equiv
\varphi(q^2)+\varphi(q^4)+3\left(\varphi(q)-\varphi(q^4)\right)
=
3\varphi(q)+\varphi(q^2)-2\varphi(q^4)
\pmod {16}.
\end{equation}
Now, employing \eqref{phi} in \eqref{e19}, we obtain
\begin{align}
2\sum_{n=0}^{\infty}F_0(2n)q^n
&\equiv 3\varphi(q)+\varphi(q^2)-2\varphi(q^4) \pmod {16} \notag\\
&\equiv3\left(1+2\sum_{r=1}^{\infty}q^{r^2}\right)
+\left(1+2\sum_{r=1}^{\infty}q^{2r^2}\right)
-2\left(1+2\sum_{r=1}^{\infty}q^{4r^2}\right) \pmod {16} \notag\\
&\equiv2+6\sum_{r=1}^{\infty}q^{r^2}
+2\sum_{r=1}^{\infty}q^{2r^2}
-4\sum_{r=1}^{\infty}q^{4r^2} \pmod {16} \notag\\
&\equiv2+6\left(
\sum_{\substack{r\geq 1\\ r\ {\rm odd}}}q^{r^2}
+\sum_{r=1}^{\infty}q^{4r^2}
\right)
+2\sum_{r=1}^{\infty}q^{2r^2}
-4\sum_{r=1}^{\infty}q^{4r^2} \pmod {16} \notag\\
&\equiv2+6\sum_{\substack{r\geq 1\\ r\ {\rm odd}}}q^{r^2}
+2\sum_{r=1}^{\infty}q^{4r^2}
+2\sum_{r=1}^{\infty}q^{2r^2} \pmod {16} \notag\\
&\equiv2+2\left(
\sum_{\substack{r\geq 1\\ r\ {\rm odd}}}q^{r^2}
+\sum_{r=1}^{\infty}q^{4r^2}
\right)
+4\sum_{\substack{r\geq 1\\ r\ {\rm odd}}}q^{r^2}
+2\sum_{r=1}^{\infty}q^{2r^2} \pmod {16} \notag\\
&\equiv2+2\sum_{r=1}^{\infty}q^{r^2}
+2\sum_{r=1}^{\infty}q^{2r^2}
+4\sum_{\substack{r\geq 1\\ r\ {\rm odd}}}q^{r^2} \pmod {16},
\end{align}which is equivalent to
\begin{equation}
\sum_{n=0}^{\infty}F_0(2n)q^n
\equiv
1
+
\sum_{r=1}^{\infty}q^{r^2}
+
\sum_{r=1}^{\infty}q^{2r^2}
+
2\sum_{\substack{r\geq1\\ r\ {\rm odd}}}q^{r^2}
\pmod 8.
\label{e23}
\end{equation}
Comparing the coefficients of the like powers of $q$ in \eqref{e23}, we arrive at  \eqref{e3}.

Again, adding \eqref{t3h} and \eqref{f1t3h}, and using \eqref{t3n}, we obtain
\begin{equation}\label{A57}
\sum_{n=0}^{\infty}\left(F_0(n)+F_1(n)\right)q^n
=
\dfrac{f_2}{f_1^2}
\equiv
1+6\sum_{r=1}^{\infty}(-1)^rq^{r^2}
+
4\sum_{r=1}^{\infty}q^{2r^2}
\pmod 8.
\end{equation}
Extracting the terms involving even powers of $q$ from both sides of \eqref{A57} and replacing $q^2$ by $q$, we obtain
\begin{equation}\label{A59}
\sum_{n=0}^{\infty}\left(F_0(2n)+F_1(2n)\right)q^n
\equiv
1+6\sum_{r=1}^{\infty}q^{2r^2}
+
4\sum_{r=1}^{\infty}q^{r^2}
\pmod 8.
\end{equation}
Employing \eqref{e23} in \eqref{A59}, we obtain
\begin{align}
\sum_{n=0}^{\infty}F_1(2n)q^n
&\equiv\left(
1+4\sum_{r=1}^{\infty}q^{r^2}
+
6\sum_{r=1}^{\infty}q^{2r^2}
\right)-\left(1+\sum_{r=1}^{\infty}q^{r^2}
+
\sum_{r=1}^{\infty}q^{2r^2}
+
2\sum_{\substack{r\geq1\\ r\ \mathrm{odd}}}q^{r^2}\right)
\pmod 8\notag\\
&\equiv
\left(
1+4\sum_{r=1}^{\infty}q^{r^2}
+
6\sum_{r=1}^{\infty}q^{2r^2}
\right)-
\left(
1+3\sum_{r=1}^{\infty}q^{r^2}
+
\sum_{r=1}^{\infty}q^{2r^2}
-
2\sum_{r=1}^{\infty}q^{4r^2}
\right)
\pmod 8
\notag\\
&\equiv
\sum_{r=1}^{\infty}q^{r^2}
+
5\sum_{r=1}^{\infty}q^{2r^2}
+
2\sum_{r=1}^{\infty}q^{4r^2}
\pmod 8.
\label{A63}
\end{align}
Now \eqref{A53} follows easily from \eqref{A63}.
\end{proof}

Proof of Corollary \ref{22} stated below is identical to the proof of Corollary \ref{cm2}, so omitted.
\begin{corollary}\label{22}
Let $M$ and $r$ be integers such that  $M\geq1$ and $0\leq r<M$. Define
\begin{equation}
Q_M=\{x^2 \pmod M:x\in\mathbb{Z}\}\cup\{2x^2 \pmod M:x\in\mathbb{Z}\}.\notag
\end{equation}
Then for $i\in\{0,1\}$,
\begin{equation}
F_i(2(Mn+r))\equiv 0 \pmod 8 \quad \text{for all } n\geq 0
\iff
r\notin Q_M.\notag
\end{equation}
\end{corollary}

\begin{remark}\it From Corollary \ref{22}, for any integer $n\geq 0$, and $i\in\{0,1\}$, the  following  congruences modulo $8$  for $F_i(n)$ can be easily obtained:
\begin{align}
&F_i(8n+6)\equiv 0 \pmod 8,\notag\\
&F_i(12n+10)\equiv 0 \pmod 8,\notag\\
&F_i(14n+6)\equiv F_i(14n+10)\equiv F_i(14n+12)\equiv 0 \pmod 8,\notag\\
&F_i(16n+6)\equiv F_i(16n+10)\equiv F_i(16n+12)\equiv F_i(16n+14)\equiv 0 \pmod 8,\notag\\
&F_i(18n+6)\equiv F_i(18n+12)\equiv 0 \pmod 8,\notag\\
&F_i(20n+6)\equiv F_i(20n+14)\equiv 0 \pmod 8,\notag\\
&F_i(24n+6)\equiv F_i(24n+10)\equiv F_i(24n+14)\equiv F_i(24n+20)\equiv F_i(24n+22)\equiv 0 \pmod 8,\notag\\
&F_i(28n+6)\equiv F_i(28n+10)\equiv F_i(28n+12)\equiv F_i(28n+20)\equiv F_i(28n+24)
\equiv F_i(28n+26)\equiv 0 \pmod 8,\notag\\
&F_i(30n+14)\equiv F_i(30n+22)\equiv F_i(30n+26)\equiv F_i(30n+28)\equiv 0 \pmod 8,\notag\\
&\hspace{5 cm}\vdots\notag
\end{align}
\end{remark}
In the remaining theorems of this section, we will use the functions  $\varepsilon(n)$ and $\mathcal{D}(n)$, which are defined, for any integer $n\geq 0$, 
\begin{equation}\label{e1}
\varepsilon(n):=
\begin{cases}
1, & \text{if } n=a(a+1) \text{ for some  integer } a\geq0,\\
0, & \text{otherwise}
\end{cases}
\end{equation}
and 
\begin{equation}\label{fg2}
\mathcal{D}(n):=
\sum_{d\mid 4n+1}\left(\dfrac{-1}{d}\right)
+
\sum_{d\mid 4n+1}\left(\dfrac{-2}{d}\right),
\end{equation}
where $\left(\dfrac{\cdot}{\cdot}\right)$ denotes the Jacobi symbol. 

\begin{theorem}
For any integer $n\ge 0$ and  $i\in\{0,1\}$, we have 
\begin{equation}\label{fg3}
F_i(4n+1)\equiv
\varepsilon(n)+2\bigl(\mathcal{D}(n)-2\varepsilon(n)\bigr)
\pmod 8.
\end{equation}
Equivalently,
\begin{equation}\label{fg4}
F_i(4n+1)\equiv
\begin{cases}
0 \pmod 8, & \text{if } \varepsilon(n)=0 \text{ and } \mathcal{D}(n)\equiv0\pmod4,\\
4 \pmod 8, & \text{if } \varepsilon(n)=0 \text{ and } \mathcal{D}(n)\equiv2\pmod4,\\
1 \pmod 8, & \text{if } \varepsilon(n)=1 \text{ and } \mathcal{D}(n)\equiv2\pmod4,\\
5 \pmod 8, & \text{if } \varepsilon(n)=1 \text{ and } \mathcal{D}(n)\equiv0\pmod4.
\end{cases}
\end{equation}
\end{theorem}
\begin{proof}
    Extracting the terms involving odd powers of $q$ from \eqref{A14f0}, dividing by $2q,$ and replacing $q^2$ by $q$,  we obtain
\begin{equation}\label{e24}
\sum_{n=0}^{\infty}F_0(2n+1)q^n
=
\dfrac{f_2^2f_8^2}{f_1^4f_4}.
\end{equation}
Employing \eqref{P7} in \eqref{e24}, we obtain
\begin{equation}\label{e25}
\sum_{n=0}^{\infty}F_0(2n+1)q^n
=
\dfrac{f_2^2f_8^2}{f_4}
\left(
\dfrac{f_4^{14}}{f_2^{14}f_8^4}
+
4q\dfrac{f_4^2f_8^4}{f_2^{10}}
\right) 
=
\dfrac{f_4^{13}}{f_2^{12}f_8^2}
+
4q\dfrac{f_4f_8^6}{f_2^8}.
\end{equation}
Extracting the terms involving even powers of $q$ from \eqref{e25} and replacing $q^2$ by $q,$ we obtain
\begin{equation}\label{e26}
\sum_{n=0}^{\infty}F_0(4n+1)q^n
=
\dfrac{f_2^{13}}{f_1^{12}f_4^2}.
\end{equation}
Employing \eqref{2t} in \eqref{e26},
we obtain
\begin{equation}\label{phps}
\sum_{n=0}^{\infty}F_0(4n+1)q^n
\equiv
\dfrac{f_2f_4^2}{f_1^4}
\equiv\dfrac{f_4^2}{f_2}
\left(\dfrac{f_2}{f_1^2}\right)^2
\equiv\psi(q^2)\phi(q)^2
\pmod 8.
\end{equation}
Using \eqref{psi} and \eqref{gp} in \eqref{phps}, we obtain
\begin{equation}\label{fg8}
\sum_{n=0}^{\infty}F_0(4n+1)q^n
\equiv
\left(\sum_{a=0}^{\infty}q^{a(a+1)}\right)
\left(\sum_{m=0}^{\infty}r_2(m)q^m\right) 
\equiv
\sum_{n=0}^{\infty}
\left(
\sum_{\substack{a\geq0\\a(a+1)\leq n}}
r_2\bigl(n-a(a+1)\bigr)
\right)q^n.
\end{equation}
By  Fundamental Theorem of Arithmetic (FTA), any integer $m\geq1$ can be expressed as $m=2^\alpha M,$ where $2\nmid M$ and $\alpha\geq0.$ Thus, from \eqref{r_2}, we obtain
\begin{equation}\label{fg8e}
r_2(m)
=
4\sum_{\substack{d\mid m\\2\nmid d}}(-1)^{(d-1)/2}
=
4\sum_{d\mid M}(-1)^{(d-1)/2} 
\equiv
4\sum_{d\mid M}1
\equiv
4d(M)
\pmod 8,
\end{equation}
where $d(M)$ is the number of positive divisors of $M.$ Again, if $p_i$ are distinct primes  and $\alpha_i\geq1$ are integers for ($i=1,\;2,\;\ldots$), then by FTA, 
$M=p_1^{\alpha_1}p_2^{\alpha_2}\cdots p_k^{\alpha_k}$
and 
\begin{equation}\label{fg10}
d(M)
=
\prod_{j=1}^{k}(\alpha_j+1).
\end{equation}
Employing \eqref{fg10} in \eqref{fg8e}, we obtain 
\begin{equation}\label{fg11}
r_2(m)\equiv
\begin{cases}
4 \pmod 8, & \text{if the odd part of }m\text{ is a perfect square},\\
0 \pmod 8, & \text{otherwise}.
\end{cases}
\end{equation}
Employing \eqref{fg11} in \eqref{fg8} and further comparing coefficients of like powers of $q$, we obtain
\begin{equation}\label{fg12}
F_0(4n+1)
\equiv
\varepsilon(n)+4~|\mathcal{E}(n)|
\pmod 8,
\end{equation}
where
$
\mathcal{E}(n)=
\left\{
a\geq0:
n-a(a+1)=2^\alpha x^2
\text{ for some integer } \alpha\geq0,\ x\geq1
\right\}
$ and $|\cdot|$ denotes the order of the set. 

Define, 
\begin{align}
\label{fg15}\mathcal{E}_1(n)
&=
\left\{
a\geq0:
n-a(a+1)=y^2
\text{ for some integer } y\geq1
\right\},\\
\intertext{and}
\label{fg16}\mathcal{E}_2(n)
&=
\left\{
a\geq0:
n-a(a+1)=2y^2
\text{ for some integer } y\geq1
\right\}.
\end{align}
Clearly, 
\begin{equation}\label{fg17}
|\mathcal{E}(n)|=|\mathcal{E}_1(n)|+|\mathcal{E}_2(n)|.
\end{equation}
Again,
$$r_2(4n+1)
=|\{(u,v)\in\mathbb Z^2:u^2+v^2=4n+1\}|.$$
This implies, for $(u,v)\in \{(u,v)\in\mathbb Z^2:u^2+v^2=4n+1\},$ 
both $u$ and $v$ cannot be even at the same time; that is, one of   $u$ and $v$ is even and the other is  odd. If $u=2a+1$ and $v=2y,$ for nonnegative integers $a,y\geq0$,  then 
\begin{equation}\label{ha3}
n-a(a+1)=y^2\iff
(2a+1)^2+(2y)^2=4n+1 .
\end{equation}
If $y=0,$ then by  \eqref{ha3},  $(u,v)=(\pm(2a+1),0),(0,\pm(2a+1))$; and if  $y\geq1,$ then  $(u,v)=(\pm(2a+1),\pm 2y),(\pm 2y,\pm(2a+1)).$
Therefore,
\begin{equation}\label{ha8}
r_2(4n+1)
=
4\varepsilon(n)+8|\mathcal E_1(n)|.
\end{equation}
Employing \eqref{r_2.4n+1} in \eqref{ha8}, we obtain
\begin{equation}
    |\mathcal E_1(n)|
=
\dfrac{1}{2}
\left(
\sum_{d\mid 4n+1}\left(\dfrac{-1}{d}\right)
-
\varepsilon(n)
\right).
\label{fg21}
\end{equation}
If we consider  $\mathcal{E}_2(n),$ then 
\begin{equation}\label{fg22}
n-a(a+1)=2y^2
\Longleftrightarrow
4n+1=(2a+1)^2+8y^2.
\end{equation}
Let $r_{1,2}(N)$ denote the number of representations of $N$ in the form $u^2+2v^2,$ where $(u,v)\in\mathbb{Z}^2.$\\ Then
\begin{equation}\label{fg23}
r_{1,2}(4n+1)=2\varepsilon(n)+4|\mathcal{E}_2(n)|.
\end{equation}
Employing \eqref{B3} in \eqref{fg23}, we obtain
\begin{equation}\label{fg25}
|\mathcal{E}_2(n)|
=
\dfrac{1}{2}
\left(
\sum_{d\mid 4n+1}\left(\dfrac{-2}{d}\right)
-
\varepsilon(n)
\right).
\end{equation}
Employing \eqref{fg21} and \eqref{fg25} in  \eqref{fg17},  we obtain
\begin{align}
|\mathcal{E}(n)|
&=
\dfrac{1}{2}
\left(
\sum_{d\mid 4n+1}\left(\dfrac{-1}{d}\right)
-
\varepsilon(n)
\right)
+
\dfrac{1}{2}
\left(
\sum_{d\mid 4n+1}\left(\dfrac{-2}{d}\right)
-
\varepsilon(n)
\right) \notag\\
&=
\dfrac{1}{2}
\left(
\sum_{d\mid 4n+1}\left(\dfrac{-1}{d}\right)
+
\sum_{d\mid 4n+1}\left(\dfrac{-2}{d}\right)
-
2\varepsilon(n)
\right) \notag\\
&=
\dfrac{1}{2}\bigl(\mathcal{D}(n)-2\varepsilon(n)\bigr).
\label{fg26}
\end{align}
Employing \eqref{fg26} in \eqref{fg12}, we obtain
\begin{equation}
F_0(4n+1)
\equiv
\varepsilon(n)+2\bigl(\mathcal{D}(n)-2\varepsilon(n)\bigr)
\pmod 8.
\end{equation}
This completes the proof of \eqref{fg3}. \\\\
Again,  from \eqref{fg26}, we note that
\begin{equation}
\mathcal{D}(n)-2\varepsilon(n)=2|\mathcal{E}(n)|,\notag
\end{equation}
Then the following four cases are possible:
\begin{align}
\varepsilon(n)=0,\quad \mathcal{D}(n)\equiv0\pmod4
&\Longrightarrow
F_0(4n+1)\equiv0\pmod8,\notag\\
\varepsilon(n)=0,\quad \mathcal{D}(n)\equiv2\pmod4
&\Longrightarrow
F_0(4n+1)\equiv4\pmod8,\notag\\
\varepsilon(n)=1,\quad \mathcal{D}(n)\equiv2\pmod4
&\Longrightarrow
F_0(4n+1)\equiv1\pmod8,\notag\\
\varepsilon(n)=1,\quad \mathcal{D}(n)\equiv0\pmod4
&\Longrightarrow
F_0(4n+1)\equiv5\pmod8.\notag
\end{align}
This proves \eqref{fg4} for the case $i=0$. It is easy to see that $F_0(2n+1)=F_1(2n+1),$ so the case  $i=1$ of \eqref{fg4} follows trivially.
\end{proof}

\begin{corollary}
For every integer $n\geq0$  and $i\in\{0,1\}$, we have
$$F_i(4n+1)\equiv0\pmod8 ~\text{if and only if } ~4n+1 \text{ is not a perfect square}$$
and 
$$\sum_{d\mid 4n+1}\left(\dfrac{-1}{d}\right)
+
\sum_{d\mid 4n+1}\left(\dfrac{-2}{d}\right)\equiv0\pmod4.$$ 
\end{corollary}
\begin{proof}
From \eqref{fg3}, we note that
\begin{align}\label{fg39}
F_i(4n+1)\equiv0\pmod8
&\Longleftrightarrow
\varepsilon(n)=0
\text{ and }
\mathcal{D}(n)\equiv0\pmod4.
\end{align}
Again, from \eqref{e1}, we note that
\begin{align}
\varepsilon(n)=1
&\Longleftrightarrow
n=a(a+1)\text{ for some integer }a\geq0 \notag\\
&\Longleftrightarrow
4n+1=4a(a+1)+1 \notag\\
&\Longleftrightarrow
4n+1=(2a+1)^2.\notag
\end{align}
Therefore,
\begin{equation}\label{fg41}
\varepsilon(n)=0
\Longleftrightarrow
4n+1 \text{ is not a perfect square}.
\end{equation} Combining \eqref{fg39}, \eqref{fg41}, and the definition of $\mathcal{D}(n)$ in  \eqref{fg2}, we obtain our desired result. 
\end{proof}

\begin{theorem}Define  $P_j=\dfrac{j(3j-1)}{2}$ for $j\in\mathbb Z$. Then for every integer $n\geq0,$  we have 
\begin{equation}\label{new1}
F_0(4n+3)\equiv
4\sum_{\substack{r,s\in\mathbb{Z}\\ n=2P_r+16P_s}}
(-1)^{r+s}
\pmod 8.
\end{equation}
Equivalently,
\begin{equation}\label{new2}
F_0(4n+3)\equiv
\begin{cases}
4 \pmod 8, & \text{if the number of representations } n=2P_r+16P_s \text{ is odd},\\
0 \pmod 8, & \text{if the number of representations } n=2P_r+16P_s \text{ is even}.
\end{cases}
\end{equation}
\end{theorem}

\begin{proof}
Extracting the terms involving odd powers of $q$ from both sides of \eqref{e25},  dividing by $q$, and replacing $q^2$ by $q$, we obtain
\begin{equation}\label{A24}
\sum_{n=0}^{\infty}F_0(4n+3)q^n
=
4\dfrac{f_2f_4^6}{f_1^8}.
\end{equation}
Using \eqref{2t}, we obtain 
\begin{equation}\label{f2f6}
\dfrac{f_2f_4^6}{f_1^8}\equiv f_2 f_{16}\pmod 2.
\end{equation}
Employing \eqref{f2f6} in \eqref{A24}, we obtain 
\begin{equation}\label{A25}
\sum_{n=0}^{\infty}F_0(4n+3)q^n
\equiv4f_2f_{16}\pmod{8}.
\end{equation}
Employing \eqref{f} in \eqref{A25}, we obtain
\begin{align}
\sum_{n=0}^{\infty}F_0(4n+3)q^n
&\equiv4f_2f_{16}\notag\\
&=
4\left(\sum_{r=-\infty}^{\infty}(-1)^rq^{2P_r}\right)
\left(\sum_{s=-\infty}^{\infty}(-1)^sq^{16P_s}\right)
\notag\\
&=
4\sum_{r=-\infty}^{\infty}
\sum_{s=-\infty}^{\infty}
(-1)^{r+s}q^{2P_r+16P_s}
\notag\\
&\label{ls}=
4\sum_{n=0}^{\infty}
\left(
\sum_{\substack{r,s\in\mathbb{Z}\\ n=2P_r+16P_s}}
(-1)^{r+s}
\right)q^n\pmod 8.
\end{align}
Comparing the coefficients of $q^n$ on both sides of \eqref{ls},  we complete the proofs of \eqref{new1} and \eqref{new2}.
\end{proof}

\section{Concluding remarks} Recall that $F(n)$ is the number of two-color integer partitions of a non-negative integer $n$ wherein odd parts may appear in two colors (red and blue) and even parts appear in one color (blue); and \( H(n)\) is the number of partitions of \( n \) counted by \(F(n)\) such that the parts of the same color do not repeat.  In addition to  proving $q$-series identities \eqref{OP1} and \eqref{2op5.6}, we proved  some congruences for the restricted versions $F_0(n)$ and $F_1(n)$ of $F(n)$, where $F_0(n)\;(\text{resp. }F_1(n))$ denotes the number of partitions of  $n$ counted by $F(n)$ wherein the number of odd parts in red color is even (resp. odd).
Interested readers may investigate divisibility properties for the following  restricted versions of $F(n)$ and $H(n)$:\\
$\bullet$ $F_2(n)$: the number of partitions of \( n \) counted by \(F(n)\) such that the number of even parts is even,\\
$\bullet$ $F_3(n)$: the number of partitions of \( n \) counted by \(F(n)\) such that the number of even parts is odd,\\
$\bullet$ $H_0(n)$: the number of partitions of \( n \) counted by \(H(n)\) such that the number of even parts is even,\\
$\bullet$ $H_1(n)$: the number of partitions of \( n \) counted by \(H(n)\) such that the number of even parts is odd,\\
$\bullet$ $H_2(n)$: the number of partitions of \( n \) counted by \(H(n)\) such that the number of parts is even,\\
$\bullet$ $H_3(n)$: the number of partitions of \( n \) counted by \(H(n)\) such that the number of parts is odd.

\section*{\bf Declarations}

\noindent\textbf{Funding}: This research did not receive funding.\\
\noindent{\bf Author Contributions.} Both authors contributed equally to this work.
   
\noindent{\bf Conflict of Interest.} The authors declare that there is no conflict of interest regarding the publication of this paper.
 
\noindent{\bf Human and Animal Rights.} The authors declare that there is no research involving human participants or animals in the context of this paper.	
 
\noindent{\bf Data Availability Statement.} Data sharing is not applicable to this paper as no datasets were generated or analyzed during the current study.	

\bibliographystyle{plain}

\end{document}